\documentclass[12pt,a4paper,psamsfonts]{amsart}
\usepackage[english]{babel}
\usepackage[T1]{fontenc}
\usepackage{fancyhdr}
\usepackage{appendix}
\usepackage{amssymb,amscd,amsxtra,calc}
\usepackage{mathrsfs}
\usepackage{cmmib57}
\usepackage{multirow}
\usepackage[all]{xy}
\usepackage{longtable}
\usepackage[colorlinks=true,linkcolor=blue,anchorcolor=blue,citecolor=blue,pagebackref,linktocpage]{hyperref}
\usepackage{cleveref}
\usepackage{tikz}
\usepackage{cite}
\usepackage{tikz,amsmath}
\tikzset{
    dyn dot/.style={circle, fill, inner sep=2.6pt},
    dyn solid/.style={solid, line width=0.5pt},
    dyn dash/.style={dashed, line width=0.5pt},
    dyn baseline/.style={baseline=(current bounding box.base)}
}

\theoremstyle{plain}
    \newtheorem{thm}{Theorem}[section]

    \newtheorem{lemma}[thm]{Lemma}
    \newtheorem{proposition}[thm]{Proposition}
    \newtheorem{question}[thm]{Question}
    \newtheorem{theorem}[thm]{Theorem}
    \newtheorem{problem}[thm]{Problem}

\theoremstyle{definition}
    \newtheorem{definition}[thm]{Definition}
     
    \newtheorem{assumption}[thm]{Assumption}
    
    \newtheorem*{notation*}{Notation and Terminology}
      
    \newtheorem{remark}[thm]{Remark}
    
\theoremstyle{remark}

\newcommand{\arxiv}[1]{\href{https://arxiv.org/abs/#1}{{\tt arXiv:#1}}}

\newcommand{\bP}{\mathbb{P}}
\newcommand{\bQ}{\mathbb{Q}}

\newcommand{\bC}{\mathbb{C}}

\newcommand{\bZ}{\mathbb{Z}}

\newcommand{\Aut}{\operatorname{Aut}}

\newcommand{\NS}{\operatorname{NS}}

\newcommand{\mstriangle}[1]{
\begin{tikzpicture}[x=0.3cm,y=0.3cm]
\draw (-0.4,-0.433) -- (1.4,-0.433);
\draw (-0.2,-0.7794) -- (0.7,0.7794);
\draw (1.2,-0.7794) -- (0.3,0.7794);
\end{tikzpicture}
}
\newcommand{\mssharp}[1]{
\begin{tikzpicture}[x=0.3cm,y=0.3cm]
\draw (-0.8,-0.5) -- (0.8,-0.5);
\draw (-0.8,0.5) -- (0.8,0.5);
\draw (-0.5,-0.8) -- (-0.5,0.8);
\draw (0.5,-0.8) -- (0.5,0.8);
\end{tikzpicture}
}

\makeatletter

\newcommand{\Rmnum}[1]{\expandafter\@slowromancap\romannumeral #1@}
\makeatother

\begin{document}

\title[Integral Zariski decompositions]{Integral Zariski decompositions on smooth projective surfaces II}
\author{Sichen Li}
\address{
School of Mathematics, East China University of Science and Technology, Shanghai 200237, P. R. China}
\email{\href{mailto:sichenli@ecust.edu.cn}{sichenli@ecust.edu.cn}}
\begin{abstract}
In this paper, we characterize three classes of smooth projective surfaces: relatively minimal elliptic surfaces with \(\chi\ge 1\), rational surfaces $X$ satisfying that \(-K_X\) is nef and $\kappa(X,-K_X)\ge1$, and projective K3 surfaces on which every integral pseudoeffective divisor admits an integral Zariski decomposition.
\end{abstract}

\maketitle
\section{Introduction}
The geometric significance of Zariski-Fujita decomposition \cite{Fujita79, Zariski62} lies in the fact that, given an integral pseudoeffective divisor $D$ on a smooth projective surface $X$ with Zariski decomposition $D=P+N$ such that $P$ and $N$ are  $\bQ$-divisors,  $P$ is nef, and one has for every sufficiently divisible integer $m\ge1$ the equality
$$
H^0(X,\mathcal O_X(mD))=H^0(\mathcal O_X(mP)).
$$ 
We say that $X$ has {\it bounded Zariski denominators} if there exists such $d(X)=m$ which is independent of the choice of $D$.
In fact, $d(X)$ is defined as the maximum of the denominators appearing in the Zariski decomposition of all integral pseudo-effective divisors.

A celebrated theorem of Bauer, Pokora, and Schmitz \cite{BPS17} states that $X$ has bounded Zariski denominators if and only if $X$ satisfies the bounded negativity conjecture (cf. \cite[Conjecture 1.1]{Bauer et al 2013}) as follows.
\begin{theorem}
\label{BPS-thm}
For a smooth projective surface $X$ over an algebraically closed field, the following two statements are equivalent:
\begin{itemize}
\item[(1)] $X$ has bounded Zariski denominators.
\item[(2)]  $X$ satisfies the bounded negativity conjecture (BNC for short), i.e., there exists an integer $b(X)\ge0$ such that $C^2\ge-b(X)$ for every curve $C\subseteq X$.
\end{itemize}
\end{theorem}
We say that a smooth projective surface $X$ satisfies \emph{integral Zariski decompositions} (IZD for short) if every integral pseudoeffective divisor $D$ on $X$ admits a Zariski-Fujiki decomposition $D=P+N$ with $P$ and $N$ integral divisors.
Note that $X$ satisfies the BNC provided that $X$ satisfies the IZD.

Let $X_n$ be a composite of blow-ups of $\bP^2$ at points $p_1,\cdots, p_n$ in a very general position, where $n\ge10$.
There are two well-known conjectures about divisors:
\begin{itemize}
	\item[(1)] the only negative curves of $X_n$ are $(-1)$-curves;
	\item[(2)] the divisor class $\sqrt{n}H-\sum_{i=1}^r E_i$ is nef.
\end{itemize}
The first one is known as the "$(-1)$-curves Conjecture" (cf. \cite[Conjecture 2.3]{CHMR13}), while the second statement is the famous Nagata Conjecture \cite[p. 772]{Nagata59}.
It is known that $(1)\Rightarrow (2)$ by  \cite[Lemma 2.4]{CHMR13}.
de Fernex  showed in \cite[Proposition 2.4]{deFernex05} that $C$ is a $(-1)$-curve if $C$ is a negative rational curve on $X_n$.
We note that every curve $C$ on $X_n$ has $C^2\ge-1$ if and only if $X_n$ satisfies the IZD (cf. \cite[Proposition 1.8]{Li19}).
Thus, it is crucially important to classify smooth projective surfaces $X$ with the IZD as follows.
\begin{problem}
\cite[Problem 2.3]{HPT15}
\label{HPT-Prb}
Classify smooth projective surfaces $X$ with the IZD.
\end{problem}
\begin{remark}
 We refer to \cite{HPT15} and \cite[Claim 2.12]{Li19} for the classification of projective K3  surfaces $X$ with Picard number $\rho(X)=2$ and the IZD.
 The characterization for smooth projective surfaces $X$ with $\rho(X)=2$, $\kappa(X)\le1$ and the IZD was established in \cite{Li19}.
Recently,  a numerical characterization of the IZD was proven in \cite{Li26}, which also showed that every fiber of fiber surfaces with the IZD is irreducible.
Thus,  the IZD is not a birational invariant, since the Hirzebruch surface \(\mathbb{F}_1\) admits the IZD (cf. \cite[Claim 2.10]{Li19}), whereas a rational Mori dream  elliptic surface fails to possess the IZD, owing to the inevitable presence of reducible fibers (cf. \cite[Theorem 5.1.4.2]{ADHL15}, see also Theorem \ref{elliptic-thm}).
\end{remark}
As the first address to Problem \ref{HPT-Prb} in this paper, we study the IZD for the cases of relatively minimal elliptic surfaces $X$ with $\chi(\mathcal O_X)\ge1$ and $b(X)>0$.
\begin{theorem}
\label{elliptic-thm}
Let $X$ be a relatively minimal  elliptic surface with $\chi(\mathcal O_X)\ge1$ and  $b(X)>0$.
Suppose $X$ satisfies the IZD.
Then $X$  falls into the following cases.
\begin{enumerate}
	\item $X$ is a Jacobian rational elliptic surface and the Mordell-Weil group $E(K)\cong E_8$.
	\item $X$ is a unnodal Halphen surface of index $m\ge2$.
	\item $X$ is a unnodal Enriques surface.
	\item $X$ is a Jacobian elliptic surface with $\kappa(X)=1$ and $\chi(\mathcal O_X)=1$.
	\item $X$ is an elliptic surface with $\kappa(X)=1$ and no  section.
 \item $X$ is an elliptic K3 surface with no  section.
 \end{enumerate}
In the Cases $(1),(2)$ and $(3)$, every negative curve $C$ on $X$ is a $(-1)$-rational curve.
\end{theorem}
\begin{remark}
To the best of our knowledge, we confirm the existence of each case in Theorem \ref{elliptic-thm}.
The existence of Case (1) may follow from \cite[Theorem 8.8]{SS19}.
It is well known that the Cases (2) and (3) are general points in the moduli spaces of Halphen surfaces and Enriques surfaces, respectively (cf. \cite[Remark 2.9]{CD12} and \cite{Cossec85}).
For Case (4), given any integer \(g>1\), there exists an elliptic surface \(\pi: X \to B\) with \(\kappa(X)=1\) and \(\chi(\mathcal{O}_X)=1\) over a smooth projective curve $B$ of genus $g$ such that \(E(K)\) is infinite which was first constructed in \cite{LL26}.
For Case (5), every minimal simply connected smooth surface with \(p_g(X)=q(X)=0\) and \(\kappa(X)=1\) -- these are known as Dolgachev surfaces -- has exactly two multiple fibers with coprime multiplicities by \cite[Chapter 2]{Dolgachev81}.
For Case (6), we will classify projective K3 surfaces satisfying the IZD provided that every automorphism is of zero entropy, see Theorem \ref{K3-thm}.
\end{remark}
It is easy to show that $X$ satisfies the IZD if every negative curve $C$ on $X$ has $C^2\ge-1$ (cf. \cite[Theorem 2.2]{BPS17}, see also Theorem \ref{Num-Thm}).
Thus, one of the essential goals of Problem \ref{HPT-Prb} is to classify surfaces with the IZD and a $(-k)$-curve where $k\ge2$.
Nevertheless, finding more examples of surfaces with the IZD seems highly constrained, but the case of projective K3 surfaces points us towards one of the accessible goals. 
\begin{definition}
\label{AFE-defn}
A hyperbolic lattice is \emph{restricted fully even} if all entries of its Gram matrix are even, its irreducible orthogonal summands contain no $U$, $A_n(n\ge2)$, $D_n , E_6, E_7$ and $E_8$, and it is not isometric to $U(k)\oplus lA_1\oplus L$ with $k, l\in \bZ_{>0}$.
\end{definition}
\begin{theorem}
\label{K3-thm}
Let $X$ be a projective K3 surface with $\rho(X)\ge3$ and $b(X)>0$.
Then the following statements hold.
\begin{enumerate}
\item[(i)] $\rho(X)\le 11$ if $X$ satisfies the IZD.
\item[(ii)] Suppose every automorphism of $X$ is of zero entropy.
	Then  $X$ satisfies the IZD  if and only if  $\NS(X)$ is isometric to one of the following 17 restricted fully even lattices:	
$$
\begin{tabular}{l l}
\hline\hline
$\varrho$ & \quad $\Lambda$ \\
\hline
3 & \quad    $(6)\oplus 2A_1$, $(2)\oplus A_1\oplus A_1(2), (2)\oplus A_1\oplus A_1(4),(8)\oplus A_1\oplus A_1(4), (16)\oplus 2A_1,$ \\ &
  \quad  $(-6,2,-2,0,4,-8), (-2k^2,0,-2,k,2,-2), \quad k\in\{4,6,8, 10, 12\},$ \\ 
   & \quad  $(-2,0,-2,k,2k,-2k^2), \quad  k\in\{2,4,6,8\}$, \\
4 &  \quad    $(8)\oplus 3A_1$, $(2)\oplus A_1\oplus 2A_1(2)$. \\
  \hline\hline
\end{tabular}
$$
	\item[(iii)] Suppose $X$ has an automorphism of positive entropy.
Then $X$ satisfies the IZD provided that $\NS(X)$ is restricted fully even.
\end{enumerate}
\end{theorem}
\begin{remark}
By  \cite[Corollary 2.9(i) and Remark 2.11]{Morrison84}, for every even lattice $N$ of signature $(1,\rho-1)$ with $\rho\le 11$, there exists a projective K3 surface $X$ with $\NS(X)\cong N$.
This provides the existence of projective K3 surfaces $X$ such that  $\NS(X)$ is restricted fully even.
\end{remark}
By Theorem \ref{K3-thm}, we ask the following question.
\begin{question}
Let $X$ be a projective K3 surface admitting an automorphism of positive entropy.
Suppose $X$ satisfies the IZD.
Is $\NS(X)$ restricted fully even?
\end{question}
Finally, we study the IZD for rational surfaces $X$ with $-K_X$ nef.
\begin{theorem}
\label{rational-thm}
Let $X$ be a smooth rational surface with $-K_X$ nef and $b(X)>0$.
Suppose $\kappa(X,-K_X)\ge1$.
Then $X$ satisfies the IZD if and only if  $X$ falls into the following cases:
\begin{enumerate}
   \item $X$ is a del Pezzo surface.
	\item $X$ is a Jacobian rational elliptic surface and $E(K)\cong E_8$.
	\item $X$ is a unnodal Halphen surface of index $m\ge2$.
\end{enumerate} 
\end{theorem}
The paper is organized as follows.
In Section \ref{Pre}, we collect two results on smooth projective surfaces with the IZD from \cite{Li26}.
We prove Theorems \ref{elliptic-thm}, \ref{K3-thm}, and \ref{rational-thm}  in Sections \ref{Sect-elliptic}, \ref{Sect-K3} and \ref{Sect-rational}, respectively.
 \section{Preliminaries}
\label{Pre}
{\bf Notation and Terminology.}
Let $X$ be a smooth projective surface over $\bC$.
\begin{itemize}
\item $K_X$ is the canonicial divisor of $X$.
\item $-K_X$ is anti-canonical divisor of $X$.
\item $q(X):=h^1(\mathcal O_X)$ is the irregularity of $X$.
\item  $\NS(X)$ is the N\'eron-Severi group of $X$.
\item By a curve on $X$, we mean a reduced and irreducible curve.
\item A negative curve on $X$ is a curve with negative self-intersection.
\item A $(-k)$-curve on $X$ is a curve $C$ with $C^2=-k<0$.
\item The intersection matrix of some divisors $C_1,\cdots,C_k$ on $X$ is $$I(C_1,\cdots,C_k)=(C_i\cdot C_j)_{1\le i,j\le k}.$$
\item  We say that $D$ is an integral pseudoeffective divisor $D$ on $X$, if it is linear equivalent to $\sum a_iC_i$ with each $a_i\in\bZ$ and each $C_i$ is a curve on $X$ and $(D\cdot C)\ge0$ for every nef curve $C$ on $X$.
\item We say that $X$ satisfies $b(X)>0$ if it has at least one negative curve.
\end{itemize}
\begin{theorem}
\cite[Theorem 1.2]{Li26}
\label{Num-Thm}
Let $X$ be a smooth projective surface. 
Then $X$ satisfies the IZD if and only if the following statements hold.
\begin{itemize}
	\item[(a)] $C^2|(C\cdot D)$ for every negative curve $C$ and every curve $D$.
	\item[(b)] for any two curves $C_1$ and $C_2$, if their intersection matrix is negative definite, then $(C_1\cdot C_2)=0$. 
\end{itemize}	
\end{theorem}
\begin{proposition}
\cite[Proposition 1.2]{Li26}
\label{irre-prop}
Let $X$ be a smooth projective surface with the IZD.
If $X$ admits a fibration $\pi: X \to B$ over a curve $B$, then every fiber is  irreducible.
\end{proposition}
\section{On elliptic surfaces with $\chi(\mathcal O_X)>0$}
\label{Sect-elliptic}
\begin{definition}
 Let $X$ be a smooth projective surface.
Let $\pi: X \to B$ be a surjective morphism with connected fibres.
We call $\pi: X\to B$ a fibration on the surface $X$ with the base curve $B$.
We say $X$ is an \emph{elliptic surface} if the general fiber of $\pi$ is an elliptic curve, and $\pi$ is \emph{relatively minimal} if no fiber contains a $(-1)$-rational curve.
 \end{definition}
For convenience,  all elliptic surfaces throughout this paper are assumed to be relatively minimal, and we will always exclude the trivial product case (cf. \cite[Example 5.6]{SS19})  by the following assumption:
\begin{assumption}
Any elliptic surface $\pi: X\to B$ is relatively minimal, and $\pi$ has a singular fibre.
Thus, we have that $\chi(\mathcal O_X)>0$ by \cite[Corollary 5.50]{SS19}.
\end{assumption}
 \begin{definition}
 Given an algebraic surface fibration $\pi: X\to B$, a \emph{section of $\pi$} is a morphism $\sigma: B \to X$ such that $\pi\circ\sigma$ is the identity map of $B$.
 An elliptic surface $\pi: X\to B$  is \emph{Jacobian} if $\pi$ has a section. 
 Let $E(K)$ be the Mordell-Weil group of $\pi: X\to B$, i.e., the group of sections, $O$ being the zero section.
 \end{definition}
 \begin{definition}
(cf. \cite[Section 2]{CD12})
\label{unnodal-Hal}
 A smooth rational projective surface $X$ is a \emph{Halphen surface} if there exists an integer $m>0$ such that the linear system $|-mK_X|$ is of dimension 1, has no fixed component, and has no base point.
 The \emph{index} of a Halphen surface is the smallest possible value for such a positive integer.
 A Halphen surface is \emph{unnodal} if it has no $(-2)$-rational curves.
 \end{definition}
 \begin{definition}
 \cite[Definition 2.3.1]{CDL25}
 An Enriques surface is \emph{unnodal} if it has no $(-2)$-rational curves.
 \end{definition}
 \begin{proposition}
\label{section-prop}
Let $X$ be an elliptic surface with $\chi(\mathcal O_X)>0$ and a section $C$.
If $X$ satisfies the IZD, then $\chi(\mathcal O_X)=-C^2=1$.
\end{proposition}
\begin{proof}
Note that $C^2=-\chi(\mathcal O_X)$ by \cite[Corollary 5.45]{SS19}.
Then by Theorem \ref{Num-Thm}(a), we have that $C^2=-1$ since $C$ is a section of the elliptic fibration.
\end{proof}
\begin{proof}[Proof of Theorem \ref{elliptic-thm}]
Let $\pi: X\to B$ be an elliptic fibration with  $\chi(\mathcal O_X)\ge1$ and $b(X)>0$.
Note that $q(X)=g(B)$ by \cite[Lemma 14 of Chapter 7]{Friedman98}, and  every fiber $F$ is irreducible by \cite[Proposition 1.2]{Li26}.
Then  every negative curve $C$ does not lie in any fiber.
Let $m_1F_1,\cdots, m_sF_s$ be the multiple fibers of $\pi$.
By Kodaira's formula for the canonical bundle of an elliptic surface (cf. \cite[Corollary V.12.3]{BHPV04} or \cite[Theorem 15 of Chapter 7]{Friedman98}), we have 
$$K_X=\delta_X F, \quad \delta_X=2g(B)-2+\chi(\mathcal O_X)+\sum(1-1/m_i).$$

 If $\delta_X<0$, then $\kappa(X)=-\infty$, $h^0(\mathcal O_X(2K_X))=0$ and $q(X)=g(B)=0$.
So $X$ is  rational by the Castelnuovo rationality criterion (cf. \cite[Theorem 2 of Chapter 10]{Friedman98}).
By \cite[Lemma 5.1.2.2]{ADHL15}, $-K_X$ is semiample since $\kappa(X,-K_X)=\kappa(X,F)=1$.
Note that $\chi(\mathcal O_X)=1$ and $\rho(X)=10$ by \cite[Proposition 7.1]{SS19}.
Take a negative curve $C$ on $X$.
$-K_X$ nef and $C^2<0$ imply that $C$ is a $(-1)$-rational curve or $(-2)$-rational curve.
Since a $(-2)$-rational curve $C$ satisfies $C\cdot K_X=0$, then it must be an irreducible component of a reducible fiber by Zariski's lemma.
This is absurd since every fiber is irreducible.
Therefore, every negative curve is a $(-1)$-rational curve.
As a result, if $\pi$ is not Jacobian, then $X$ is a unnodal Halphen surface of index $m\ge2$.
Now we assume that $\pi$ is Jacobian.
If $\mathrm{MW}(\pi)$ is finite, then $X$ is an extremal Jacobian elliptic surface.
Then $\pi$ has at least  a reducible fiber by the classification result  (\cite{MP86}, see also \cite[Theorem 5.1.4.2(2)]{ADHL15}).
This is absurd since $\pi$ has no reducible fiber.
If $\mathrm{MW}(\pi)$ is infinite, then  $E(K)\cong E_8$ by \cite[Corollary 6.7 and Theorem 7.4(i)]{SS19}.
Now we may assume that $\pi$ is not Jacobian.
Then $X$ is a unnodal Halphen surface of index $m\ge2$.

 If $\delta_X=0$, then $K_X\equiv0$ and $C$ is a $(-2)$-rational curve.
By the Enriques-Kodaira classification (cf. \cite[Chapter VI, Table 10, p. 244]{BHPV04}), $X$ is a K3 surface, an Enriques  surface, an abelian surface or a bi-elliptic surface.
Therefore, $X$ is an Enriques surface or  a K3 surface  since abelian surfaces and bi-elliptic surfaces have no any rational curves.
If $X$ is an Enriques surface, then $X$ admits an elliptic fibration with a reducible fiber provided that $X$ contains a $(-2)$-rational curve by \cite[Theorem 4]{Cossec85}.
Thus, $X$ is a unnodal Enriques surface.
Now we assume that $X$ is a K3 surface.
Note that $\chi(O_X)=2$.
Therefore, $X$ is an elliptic K3 surface with no section by Proposition \ref{section-prop}.

 If $\delta_X>0$, then $\kappa(X)=1$.
 If $\pi$ is Jacobian, then $\chi(\mathcal O_X)=1$ by Proposition \ref{section-prop}.
 If $\pi$ is not Jacobian, then $X$ is an elliptic surface with $\kappa(X)=1$ and no  section.
\end{proof}
\section{Proof of Theorem \ref{K3-thm}}
\label{Sect-K3}
\subsection{Lattices}
In the context of K3 surfaces,  a \emph{lattice} is a finitely generated free abelian group $\Lambda$ equipped with an integer symmetric bilinear form $\beta\colon \Lambda \times \Lambda \to \mathbb Z$;
we will mostly write $w \cdot w' := \beta(w, w')$ and $w^2 := \beta(w, w)$. Such a lattice $\Lambda$ is called
\begin{enumerate}
    \item Nondegenerate if $\beta$ is of rank $\operatorname{rk}(\beta) = \operatorname{rk}(\Lambda)$,
    \item Even if $w^2 \in 2\mathbb Z$ holds for all $w \in \Lambda$,
    \item Hyperbolic if $\beta$ has signature $(1, \varrho - 1)$ where $\varrho = \operatorname{rk}(\beta)$.
\end{enumerate}
 There is the lattice $U = \mathbb Z^2$ of rank $2$ with the intersection form given by the Gram matrix
\[
\begin{pmatrix}
0 & 1 \\
1 & 0
\end{pmatrix}.
\]
Moreover, there are the lattices $A_n$, $D_n$, $E_6$, $E_7$, and $E_8$. They are characterized by means of their Coxeter-Dynkin graphs: the vertices are the simple roots of the lattice, that is, the finitely many elements $w$ with $w^2 = -2$ and two roots $w, w'$ are joined by an edge if $w \cdot w' = 1$ (if they are not joined then $w \cdot w' = 0$ holds).

~~
\tikzset{
    dyn dot/.style={circle, fill, inner sep=1.2pt},
    dyn solid/.style={solid, line width=0.5pt},
    dyn dash/.style={dashed, line width=0.5pt},
    dyn baseline/.style={baseline=(current bounding box.base)}
}

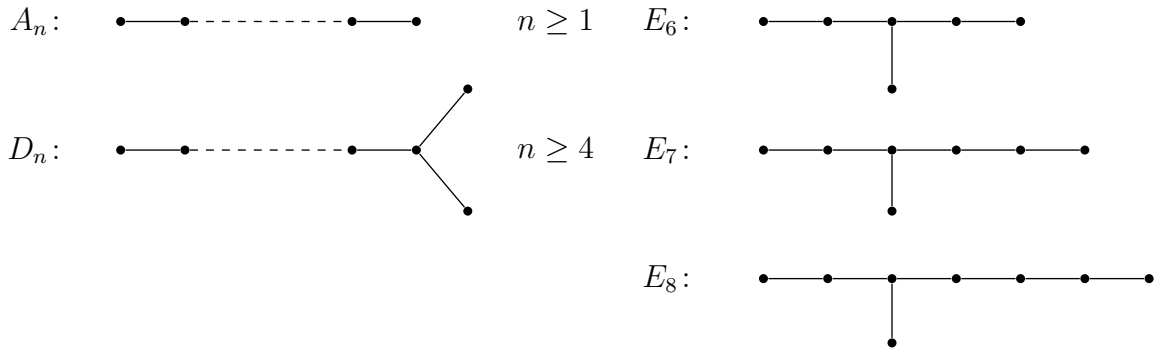
\begin{tikzpicture}[dyn baseline, scale=0.85]
\node[anchor=east] at (-0.6, 2.0) {$A_n\colon$};
\node[dyn dot] (A1) at (0, 2.0) {};
\node[dyn dot] (A2) at (1.0, 2.0) {};
\node[dyn dot] (A3) at (3.6, 2.0) {};
\node[dyn dot] (A4) at (4.6, 2.0) {};
\draw[dyn solid] (A1)--(A2);
\draw[dyn dash] (A2)--(A3);
\draw[dyn solid] (A3)--(A4);
\node[anchor=west] at (6, 2.0) {$n\ge 1$};

\node[anchor=east] at (-0.6, 0) {$D_n\colon$};
\node[dyn dot] (D1) at (0, 0) {};
\node[dyn dot] (D2) at (1.0, 0) {};
\node[dyn dot] (D3) at (3.6, 0) {};
\node[dyn dot] (D4) at (4.6, 0) {};
\node[dyn dot] (Dup) at (5.4, 0.95) {};
\node[dyn dot] (Ddn) at (5.4, -0.95) {};
\draw[dyn solid] (D1)--(D2);
\draw[dyn dash] (D2)--(D3);
\draw[dyn solid] (D3)--(D4);
\draw[dyn solid] (D4)--(Dup);
\draw[dyn solid] (D4)--(Ddn);
\node[anchor=west] at (6, 0) {$n\ge 4$};

\def\xshift{9.2}
\node[anchor=east] at (\xshift, 2.0) {$E_6\colon$};
\node[dyn dot] (E6a) at (\xshift+0.8, 2.0) {};
\node[dyn dot] (E6b) at (\xshift+1.8, 2.0) {};
\node[dyn dot] (E6c) at (\xshift+2.8, 2.0) {};
\node[dyn dot] (E6d) at (\xshift+3.8, 2.0) {};
\node[dyn dot] (E6e) at (\xshift+4.8, 2.0) {};
\node[dyn dot] (E6v) at (\xshift+2.8, 0.95) {};
\draw[dyn solid] (E6a)--(E6b)--(E6c)--(E6d)--(E6e);
\draw[dyn solid] (E6c)--(E6v);

\node[anchor=east] at (\xshift, 0) {$E_7\colon$};
\node[dyn dot] (E7a) at (\xshift+0.8, 0) {};
\node[dyn dot] (E7b) at (\xshift+1.8, 0) {};
\node[dyn dot] (E7c) at (\xshift+2.8, 0) {};
\node[dyn dot] (E7d) at (\xshift+3.8, 0) {};
\node[dyn dot] (E7e) at (\xshift+4.8, 0) {};
\node[dyn dot] (E7f) at (\xshift+5.8, 0) {};
\node[dyn dot] (E7v) at (\xshift+2.8, -0.95) {};
\draw[dyn solid] (E7a)--(E7b)--(E7c)--(E7d)--(E7e)--(E7f);
\draw[dyn solid] (E7c)--(E7v);

\node[anchor=east] at (\xshift, -2.0) {$E_8\colon$};
\node[dyn dot] (E8a) at (\xshift+0.8, -2.0) {};
\node[dyn dot] (E8b) at (\xshift+1.8, -2.0) {};
\node[dyn dot] (E8c) at (\xshift+2.8, -2.0) {};
\node[dyn dot] (E8d) at (\xshift+3.8, -2.0) {};
\node[dyn dot] (E8e) at (\xshift+4.8, -2.0) {};
\node[dyn dot] (E8f) at (\xshift+5.8, -2.0) {};
\node[dyn dot] (E8g) at (\xshift+6.8, -2.0) {};
\node[dyn dot] (E8v) at (\xshift+2.8, -3.0) {};
\draw[dyn solid] (E8a)--(E8b)--(E8c)--(E8d)--(E8e)--(E8f)--(E8g);
\draw[dyn solid] (E8c)--(E8v);
\end{tikzpicture}

~~

Given two lattices $\Lambda_1$, $\Lambda_2$, we denote by $\Lambda_1 \oplus \Lambda_2$ the direct sum lattice and by $n\Lambda$ the $n$-fold direct sum of $\Lambda$. 
Moreover, for a lattice $\Lambda$ and an integer $n$, we write $\Lambda(n)$ for $\Lambda$ with the stretched product $n\langle~, ~ \rangle$. Specially, we denote by $(n)$ the lattice $\mathbb Z$ with $a \cdot b = nab$.
\begin{definition}
\label{gram-defn}
For a symmetric matrix $(a_{ij})$, we use its lower left entries $$(a_{11}, a_{21},a_{22},\cdots, a_{nn})$$ to denote the lattice with Gram matrix equal to it.
For example, by $(-6,0,-2,0,4,-4)$ we mean the rank 3 lattice with Gram matrix
$\begin{pmatrix}
6 & 0 & 0 \\
0 & -2 & 4\\
0 & 4 & -4 \\
\end{pmatrix}.$
\end{definition}
\subsection{Dynamical degrees}
Let $f: X\to  X$ be a surjective endomorphism of an $n$-dimensional normal projective variety $X$. 
Then the {\it first dynamical degree} of $f$ (cf. \cite{Dang20, DS05, Truong20}) is defined as 
$$
 d_1(f):=\lim_{s\to\infty}((f^s)^*H\cdot H^{n-1})^{1/s},
$$
where $H$ is an ample Cartier divisor on $X$.
Dinh and Sibony's result states that this limit exists and is independent of the choice of the ample divisor $H$.
We say $f$ is of positive entropy if $d_1(f)>1$, otherwise $f$ is of zero entropy.
In addition, we say the automorphism group $\Aut(X)$ of $X$ is of zero entropy if every automorphism is of zero entropy.

Now we recall a characterization of positive entropy of automorphism of K3 surfaces.
\begin{theorem}
(cf. \cite[Theorem 1.4]{Oguiso07})
Let $X$ be a K3 surface, $G\le \Aut(X)$ and $g\in \Aut(X)$.
Then the following statements hold.
\begin{enumerate}
	\item $G$ is of zero entropy if and only if $G$ is finite or $G$ makes an elliptic fibration $\pi: X\to\bP^1$ stable.
\item $g$ is of positive entropy if and only if $g$ has a Zariski dense orbit, i.e., there is a point $x\in X$ such that the set $\{g^n(x)| n\in\bZ\}$ is Zariski dense in $X$.
\end{enumerate}
\end{theorem}
\subsection{Proof of Theorem \ref{K3-thm}}
\begin{lemma}
\label{num-K3-lem}
Let $X$ be projective K3 surface with $b(X)>0$.
Then $X$ satisfies the IZD if and only if the following statements hold.
\begin{itemize}
\item[(a)]	for every $(-2)$-curve $C$ and every curve $D(\ne C)$, we have $C\cdot D\in2\bZ_{\ge0}$.
\item[(b)] for any two  curves $C_1$ and $C_2$, if $I(C_1,C_2)$ is negative definite, then $(C_1\cdot C_2)=0$.
\end{itemize}
\end{lemma}
\begin{proof}
Let $C$ be a negative curve.
Note that $K_X\equiv 0$.
Then $K_X\cdot C=0$ and $C^2<0$ implies that $C$ is a $(-2)$-rational curve by the adjunction formula.
So the proof follows from Theorem \ref{Num-Thm}.
\end{proof}
\begin{proposition}
\label{even-lattice-prop}
Let $X$ be a projective K3 surface with $\rho(X)\ge2$.
Then $X$ satisfies the IZD  provided that $\NS(X)$ is restricted fully even.
\end{proposition}
\begin{proof}
Since $\NS(X)$ is restricted fully even, the intersection number of any two generator elements of $\NS(X)$ is an even integer.
As a result, for every $(-2)$-curve $C$ and every curve $D(\ne C)$, we have $C\cdot D\in2\bZ_{\ge0}$.
We note that (a) of  Lemma \ref{num-K3-lem} implies (b) of Lemma \ref{num-K3-lem} by \cite[Proof of Proposition 2.1]{Li26}.
Then $X$ satisfies the IZD by Lemma \ref{num-K3-lem}.
\end{proof}
\begin{proposition}
\label{odd-prop}
Let $X$ be a projective K3 surface.
Let $\NS(X)$ be isometric to $$(a_{11},a_{21}, a_{22},\cdots, a_{nn})$$
with a basis $e_1,\cdots, e_n$.
Then the following statements hold.
\begin{enumerate}
	\item If  $a_{ii}=-2$, then either $e_i$ or $-e_i$ is a $(-2)$-curve.
	\item If   $a_{ii}=-2$ and $a_{ij}$ is an odd integer with $i\ne j$, then $X$ does not satisfy the IZD.
	\item Suppose  $\NS(X)$ is isometric to $U(k)\oplus A_1\oplus L$.
	Then $X$ does not satisfy the IZD.
\end{enumerate}
\end{proposition}
\begin{proof}
 (1). By Riemann-Roch theorem, either $e_i$ or $-e_i$ is a curve since $e_i^2\ge-2$.

(2). Now we may assume that $e_i$ is a $(-2)$-curve and $e_i\cdot e_j$ is an odd integer.
Thus, $X$ does not satisfy the IZD by Lemma \ref{num-K3-lem}.

(3). 
Suppose \(\NS(X)\cong U(k)\oplus A_1\oplus L\).
 By \cite[Proposition 11.1.3]{Huybrechts16},  \(\NS(X)\) contains a primitive isotropic vector, hence $X$ admits an elliptic fibration \(\pi\colon X\to\mathbb P^1\). 
 Let $D$ be the \((-2)\)-rational curve corresponding to the summand \(A_1\). 
Then $D\cdot F=0$, where $F$ is a fiber of $\pi$.
By Zariski's lemma, $D$ must be an irreducible component of a reducible fiber.
This is absurd by Proposition \ref{irre-prop}.
Thus, $X$ does not satisfy the IZD.
\end{proof}
\begin{proof}[Proof of Theorem \ref{K3-thm}]

(i).  We  assume that $X$ satisfies the IZD.
If $\rho(X)\ge12$, then there exists an embedding $U\hookrightarrow	\NS(X)$ by \cite[Corollary 13.3.8]{Huybrechts16}.
Then there exists a Jacobian elliptic fibration by \cite[Lemma 3.1(i)]{Kondo89}.
Note that a section $C$ has $C^2=-2$.
This is absurd by Lemma \ref{num-K3-lem}.
Thus, $\rho(X)\le 11$.

(iii) the proof  follows from  Proposition \ref{even-lattice-prop}.

(ii) The proof of sufficiency for (ii) follows from Proposition \ref{even-lattice-prop}.

Now suppose that $X$ satisfies the IZD.
Note that $\Aut(X)$ is finite if and only if $X$ is a Mori dream surface (cf. \cite[Theorem 5.1.5.1]{ADHL15}).
If $\Aut(X)$ is finite,  by \cite[Theorem 5.1.5.3]{ADHL15}, We have a classification of Mori dream K3 surfaces.
If $\Aut(X)$ is infinite and every automorphism is of zero entropy, we  have a  classification as in \cite[Appendix A]{Yu25}.

We go through the complete list of zero-entropy N\'eron-Severi lattices given in \cite[Theorem 5.1.5.3]{ADHL15} and \cite[Appendix A]{Yu25}, and filter them using the IZD criteria from Proposition \ref{odd-prop}.
This allows us to rule out lattices containing $U$, $U(k)\oplus lA_1$ ($k,l>0$), $A_n$ ($n\ge 2$), $D_n$, $E_6$, $E_7$, or $E_8$ as orthogonal direct summands, as well as lattices  admitting a $(-2)$-rational curve whose intersection with some divisor is odd by Proposition \ref{odd-prop}.
Let $e_1,e_2,e_3,\dots,e_{\rho(X)}$ be a basis for the N\'eron-Severi lattice $N$. We then sequentially check each Picard number $\rho(X)\in [3,11]$ to obtain our final classification.

Now we assume that $\rho(X)=3$.
If $\Aut(X)$ is finite,  by \cite[Theorem 5.1.5.3]{ADHL15}, $N$ is one of the following lattices:
$$(2e_1+e_3, e_2, 2e_3), \quad (ke_1, e_2, e_3), k\in\{4,5,6,7,8, 10,12\}, $$ $$
(e_1, ke_2, e_3), k\in \{2,3,4,5,6,9\}, \quad (e_1,e_2,ke_3), k\in\{1,2,3,4,6,8\},$$
where the intersection matrix of $e_1,e_2,e_3$ is $(-2,0,-2,1,2,-2)$ (see Definition \ref{gram-defn} for this notation) and $(6)\oplus 2A_1$.
Note that $(6)\oplus 2A_1$ is restricted fully even and 
\begin{enumerate}
   \item  $(6)\oplus 2A_1$ is restricted fully even,
	\item  $I(2e_1+e_3, e_2, 2e_3)=(-6,2,-2,0,4,-8)$  is restricted fully even.
	\item $I(ke_1,e_2,e_3)=(-2k^2,0, -2, k,2,-2)$  is  restricted fully even if and only if $k\in \{4,6,8,10,12\}$.
\item  $I(e_1,e_2, ke_3)=(-2,0,-2,k,2k,-2k^2)$ is restricted fully even if and only if $k\in\{2,4,6,8\}$.
\end{enumerate}
By Proposition \ref{odd-prop}, we note that $X$ does not satisfy the IZD provided that  $N$ is one of the following lattices:
\begin{itemize}
	\item[(a)] $I(ke_1, e_2, e_3)=(-5k^2, 0, -2, k, 2,-2)$, $k\in\{5, 7\}$,
	\item[(b)] $I(e_1,ke_2,e_3)=(-2,0, -2k^2,1, 2k,-2)$, $k\in\{2,3,4,5,6,9\}$,
	\item[(c)] $I(e_1,e_2,ke_3)=(-2,0,-2, k, 2k,-2k^2)$, $k\in \{1,3\}$,
\end{itemize}
since $(e_1\cdot e_3)$ is an odd integer in all cases (a), (b) and (c).

If  $\Aut(X)$ is infinite, by \cite[Appendix A]{Yu25}, $N$ is one of the following cases:
$$
(2)\oplus A_1\oplus A_1(2), (2)\oplus A_1\oplus A_1(4), (16)\oplus 2A_1,(8)\oplus A_1\oplus A_1(4),$$
(which are restricted fully even), and 
$$ (0, 2,2,0,1, -14), (0,3,4)\oplus A_1(9).
$$
Thus, it suffices to show that $X$ does not satisfy the IZD provided that $N$ is either $(0,2,2,0,1,-14)$ or $(0,3,4)\oplus A_1(9)$.

We  assume that $N=(0, 2,2,0,1, -14)$.
Note that $$e_1^2=e_1\cdot e_3=0, e_1\cdot e_2=2,e_2^2=2,e_2\cdot e_3=1,  e_3^2=-14.$$
Take a $(-2)$-curve $r=ae_1+be_2+ce_3, a,b,c\in\bZ$.
$r^2=-2$ implies that $4ab+2b^2+2bc-14c^2=-2$.
As a result, $2ab+b^2+bc-7c^2=-1$.
Note that $$2ab+b^2+bc-7c^2\equiv b^2+bc-c^2\equiv 1(\mathrm{mod}~ 2).$$
Since $x^2\equiv x(\mathrm{mod}~ 2)$, then $$b(1+c)\equiv 1+c(\mathrm{mod}~ 2).$$
This implies that $0\equiv 1(\mathrm{mod}~ 2)$ if $b$ and $c$ are even.
Therefore, one of $b$ and $c$ is odd.
Note that $r\cdot e_1=2b+c, r\cdot e_2=2a+2b+c$ and $r\cdot e_3=b-14c$.
Thus, $r\cdot e_i$ is odd for some $i\in\{1,2,3\}$.
As a result, $X$ does not satisfy the IZD by Proposition \ref{odd-prop}.

Now we assume that $N=(0,3,4)\oplus A_1(9)$.
Note that $$e_1^2=e_1\cdot e_3=e_2\cdot e_3=0, e_2^2=4, e_1\cdot e_2=3, e_3^2=-18.$$
Let $r=ae_1+be_2+ce_3$ be the class of a $(-2)$-curve.
Then $r^2=-2$ implies that 
\begin{equation}
\label{A_{1,9}-eqI}
6ab+4b^2-18c^2=-2.
\end{equation}
We  first show that the class of $e_2$ is ample.
It suffices to show that $(e_2\cdot r)>0$ for every $(-2)$-rational curve $r$ by \cite[Proposition 2.1.4]{Huybrechts16}.
Note that 
$$
r\cdot e_2=3a+4b.
$$
If $r\cdot e_2=0$, then substitute \(3a = -4b\)  into (\ref{A_{1,9}-eqI}), we have $18c^2+4b^2=2$, which is absurd since $b$ and $c$ are non-negative integers.
Therefore, $e_2$ is an ample class.
Let $v=e_2-e_1$.
Then $v^2=-2$.
We claim that $v$ is a $(-2)$-rational curve.
Note that $v\cdot e_2=1>0$ implies that $v$ is  effective
If $v=v_1+v_2$ such that $v_1$ and $v_2$ are effective, then $v\cdot e_2\ge2$ since $e_2$ is ample.
This is absurd since $v\cdot e_2=1$.
Thus, $v$ is irreducible.
As a result, $v$ is  a $(-2)$-rational curve.
Since $v\cdot e_2=1$, $X$ does not satisfy the IZD by Proposition \ref{odd-prop}.

Now we assume that $4\le \rho(X)\le 11$.
By \cite[Theorem 5.1.5.3]{ADHL15} and \cite[Appendix]{Yu25}, and Propositions \ref{even-lattice-prop} and \ref{odd-prop}, $N$ is either $(8)\oplus 3A_1$ or $(2)\oplus A_1\oplus 2A_1(2)$.
Therefore, we complete the proof of Theorem \ref{K3-thm}.
\end{proof}
\section{On rational surfaces with nef anti-canonical divisors}
\label{Sect-rational}
\begin{definition}
Let $X$ be a smooth projective surface.
\begin{enumerate}
	\item $X$ is a \emph{del Pezzo surface}, if $-K_X$ is ample.
	\item $X$ is a \emph{weak del Pezzo surface}, if $-K_X$ is nef and big. 
\end{enumerate}
\end{definition}
\begin{definition}
\cite[Definition 5.2.1.6]{ADHL15}
Let $X_0=\bP^2,\cdots, X_r=X$ be a sequence of blow-ups of points $p_i\in X_{i-1}$, where $1\le r\le 8$.
We say that $p_1,\cdots, p_r$ are in \emph{almost general position} if
\begin{itemize}
	\item No four of them are mapped to the same line in $\bP^2$.
	\item No seven of them are mapped to the same conic in $\bP^2$.
	\item The total transform $E_i$ of the exceptional divisor over $p_i\in X_{i-1}$ is an extended $(-1)$-curve for any $i$.
\end{itemize}
\end{definition}
\begin{proof}[Proof of Theorem \ref{rational-thm}]
($\Leftarrow$) 
Let $\Delta$ be the determinant of the intersection form on the NS lattice of $X$.
Then $|\Delta|=1$ since $X$ is a rational surface.
Note that every negative curve on $X$ is a $(-1)$-rational curve if  $X$ falls into one of the cases (1), (2) and (3).
Thus, $X$ satisfies the IZD by \cite[Theorem 2.2]{BPS17}.

($\Rightarrow$) We assume that $X$ satisfies the IZD.
If $\kappa(X,-K_X)=2$, then $X$ is a weak del Pezzo surface.
By \cite[Theorem 5.2.1.7]{ADHL15}, up to isomorphy $X$ is either $\bP^1\times \bP^1$ or the Hirzebruch $F_2$ or a blow-up of $\bP^2$ at $0\le r\le 8$ points in almost general position.
If $X$ is $\bP^1\times \bP^1$ or $F_2$, then $\rho(X)=2$.
This is absurd by \cite[Claim 2.10]{Li19}.
Then $\pi: X \to\bP^2$ is a blow-up of $\bP^2$ at $0\le r\le 8$ points $p_1,\cdots, p_r$ in almost general position.
Let $H$ be the total transform of any line in $\bP^2$.
If $-K_X$ is not ample, then there exists a $(-2)$-rational curve $C$ by the adjunction formula.
Let $r\ge2$.
By \cite[Remark 5.2.1.11]{ADHL15}, $C$ is either of the form $E_i-E_{i+1}$, or it is linearly equivalent to one of the following (modulo permutation of the indices):
$$H-\sum_{k=1}^3E_k, \quad 2H-\sum_{k=1}^6E_k, \quad 3H-2E_1-\sum_{k=2}^8 E_k.$$
In all cases, $C\cdot E_{i+1}=1$ or $C\cdot E_3=1$.
This is absurd by Theorem \ref{Num-Thm}(a).
Thus  $X$ is a del Pezzo surface.
If $\kappa(X,-K_X)=1$, then $|-mK_X|$ with $m>0$ induces a relatively minimal elliptic fibration. 
By Theorem \ref{elliptic-thm}, $X$ is either  a Jacobian rational elliptic surface and $E(K)\cong E_8$ or a unnodal Halphen surface of index $m\ge2$.
\end{proof}

\subsection*{Acknowledgements}
The author would like to thank Xun Yu for answering questions.

\end{document}